\documentclass[12pt]{amsart}
\usepackage{hyperref}
\usepackage{subfig, tikz}
\usepackage{amsmath}
\usepackage{amsthm}
\usepackage{mathrsfs}
\usepackage{amssymb}
\usepackage{tikz-cd}
\usepackage{comment}
\usepackage{mathtools}
\usepackage{pgfplots}
\usepackage{pgfplotstable}
\pgfplotsset{compat=1.7}
\usepackage{manfnt}
\calclayout
\usepackage{xcolor}
\newcommand{\PP}{\mathbb{P}}

\newcommand{\RR}{\mathbb R}

\newcommand{\OO}{\mathcal O}

\newcommand{\Eff}{\overline{\mathrm{Eff}}}
\newcommand{\Psef}{\overline{\mathrm{Eff}}}

\newcommand{\Nef}{\mathrm{Nef}}

\newcommand{\ra}{\rightarrow}

\newcommand{\BPF}{\mathrm{BPF}}

\newcommand*{\shom}{\mathcal{H}\kern -.5pt om} 

\newcommand{\twopartdef}[4]
{
	\left\{
		\begin{array}{ll}
			#1 & \mbox{if } #2 \\
			#3 & \mbox{if } #4
		\end{array}
	\right.
}
\newtheorem{theorem}{Theorem}[section]

\newtheorem{lemma}[theorem]{Lemma}
\newtheorem{proposition}[theorem]{Proposition}
\newtheorem{corollary}[theorem]{Corollary}

\newtheorem{conjecture}[theorem]{Conjecture}

\theoremstyle{definition}
\newtheorem{question}[theorem]{Question}
\newtheorem{remark}[theorem]{Remark}

\newtheorem{definition}[theorem]{Definition}

\title{Effective cycles on universal hypersurfaces}
\author[G.  Smith]{Geoffrey Smith}
\address{Department of Mathematics, Statistics and CS \\University of Illinois at Chicago, Chicago, IL 60607}
\email{geoff@uic.edu}
\begin{document}
\begin{abstract}
We study the effective cones of cycles on universal hypersurfaces on a projective variety $X$, particularly focusing on the case of universal hypersurfaces in $\PP^n$.
We determine the effective cones of cycles on the universal conic over $\PP^2$. We also determine every effective cone of cycles of dimension at most 6 or equal to 10 on any universal plane curve. 
\end{abstract}
\maketitle
\section{Introduction and statement of results}
The effective cones of divisors and curves on a smooth projective variety have been extensively studied.
The pseudoeffective cones of cycles of intermediate dimension are much less well-understood, and have been computed only in a handful of cases \cite{DELV11, Ful11,  Ott15, CC15, CLO16, Kop18, BKLV19, PPS21}.

In this article, we extend this study to universal hypersurfaces.
Let $X$ be a smooth projective variety of dimension $n$.
Given an effective divisor class $H$ on $X$, let $Y_H\cong \PP(H^0(X,H)^*)$ be the space of all divisors of class $H$ on $X$. Then the \emph{universal hypersurface} $X_H$ is the incidence correspondence $X_H \subset X\times Y_H$ defined by
\[
X_H=\{(p,D)\in X\times Y_H\vert p\in D\}.
\]
If $H$ is basepoint free, the projection $\pi_1:X_H\ra X$ is a projective bundle over $X$.
This article is devoted to the study of effective cycles on the varieties $X_H$, especially when $X$ is a projective space.

For many varieties $X$, given $i<\dim(X)$, we can determine the nef cones $\Nef^i(X_H)$---and hence the pseudoeffective cones $\Psef_i(X_H)$---explicitly. Let $\xi$ be the pullback to $X_H$ of the hyperplane class on $Y_H$. 
\begin{theorem}\label{lowCodimension}
Let $X$ be smooth, and let $H$ be a divisor class on $X$  such that for every $i<\dim(X)$, $\Psef_i(X)$ is the closure of a cone spanned by the numerical classes of cycles supported on a divisor of class $H$.  If $0\leq i< \dim(X)$, then \[\Nef^i(X_H)=\sum_{0\leq j\leq i} \xi^{i-j}\pi_1^*(\Nef^j(X)) .\]
\end{theorem}
Since $\Psef_i(X_H)$ is the dual cone of $\Nef^i(X_H)$, Theorem \ref{lowCodimension} also determines $\Eff_i(X_H)$ for $i<\dim(X)$ whenever its hypotheses are satisfied. Both $\Psef_i(X_H)$ and $\Nef^i(X_H)$ are defined in Section \ref{fl}.
 
Theorem \ref{lowCodimension} applies if $X$ is $\PP^n$ or any other smooth variety $X$ with $\dim(N_i(X))=1$ for all $0\leq i \leq \dim(X)$.
It also applies if $X$ is a del Pezzo surface and $H+K_X$ is effective. 
Finally, it applies to any smooth variety $X$ for which $\Psef_i(X)$ is generated by finitely many effective classes for all $i$, as long as $H$ is sufficiently ample. (See Corollary \ref{mostOfTheStableCone} for a precise statement.)

For $i\geq \dim(X)$ the effective cones of cycles on $X_H$ become more complicated.
In this paper, we will study these cones for the universal hypersurface of degree $d$ in $\PP^n$, which we denote $X_{n,d}$. If we let $H$ be the pullback to $X_{n,d}$ of the hyperplane class on $\PP^n$, we have
\[
N^\bullet(X_{n,d})=A^\bullet(X_{n,d})_\RR=\frac{\RR[\xi, H]}{(H^{n+1}, P(H,\xi))},
\]
where $P$ is a homogeneous polynomial of degree $\binom{n+d}{d}-1$. 

If $d=1$, then $X_{n,1}$ is a rational homogeneous variety on which the nef and effective cones of cycles coincide. We have the following result---a special case of the classification of the effective cycles on rational homogeneous varieties given in \cite{Cos18}.
\begin{proposition}[{\cite[Proposition 2.20]{Cos18}}]\label{hyperplaneCone}
If $i<n$, then
\[
\Nef^i(X_{n,1})=\Psef^i(X_{n,1})=\langle H^{i},H^{i-1}\xi,\ldots, \xi^i\rangle.
\]

Dually, given $0\leq j\leq n-1$ and a $j$-plane $\Lambda\subset \PP^n$, let $Z_i\subset X_{n,1}$ be the $(n-1)$-cycle consisting of all $(p,D)$ with $p\in \Lambda_j$ and $\Lambda_j\subset D$. Then, for any $0\leq i\leq n-1$, we have
\[
\Nef_i(X_{n,1})=\Psef_i(X_{n,1})=\langle \xi^{n-i-1}[Z_0], \xi^{n-i-1}[Z_1],\ldots, \xi^{n-i-1}[Z_i]\rangle.
\]
\end{proposition}
We also determine all cones of pseudoeffective cycles on the universal conic $X_{2,2}$. To write them down, we need to define five new cycles.
Fix a point $p_0\in \PP^2$ and a line $\ell_0\subset \PP^2$ containing $p_0$.
 Let $Z_{4,1},Z_{4,2},Z_{4,3}\subset X_{2,2}$ be 4-cycles defined by
\begin{align*}
Z_{4,1}&=\overline{\{(p,D)\in X_{2,2}\vert p_0\in D, p\in (D\cap \ell_0)\setminus \{p_0\}\}}\\
Z_{4,2}&=\{(p,\ell_1+ \ell_2)\in X_{2,2}\vert  p,p_0\in \ell_1 \} \\
Z_{4,3}&= \{(p,\ell_1+ \ell_2)\in X_{2,2}\vert p\in \ell_1; p_0\in \ell_2\}.
\end{align*}
Typical examples of conics in these 4-cycles are depicted in Figure \ref{figure1} below.
\begin{figure}
\centering{
\subfloat[$H^2$]{
\begin{tikzpicture}
  \node at (2, 1.5)   {D};
\draw[red, thick] plot[smooth,domain=-2:2] (\x, {\x*\x-1});
\filldraw[black] (-1,0) circle (2pt) node[black, anchor=west]{$p=p_0$};
\end{tikzpicture}
}
\hspace{5em}
\subfloat[$Z_{4,1}$]{
\begin{tikzpicture}
  \node at (2, 1.5)   {D};
\draw[gray, thick] (-3,0) -- (3,0) node[black, anchor=north] {$\ell_0$};
\draw[red, thick] plot[smooth,domain=-2:2] (\x, {\x*\x-1});
\filldraw[black] (-1,0) circle (2pt) node at (-1.2,-.25) {$p_0$};
\filldraw[blue] (1,0) circle (2pt) node[black, anchor=south]{$p$};
\end{tikzpicture}
}

\subfloat[$Z_{4,2}$]{
\begin{tikzpicture}
\draw[red, thick] (-2,-1) -- (1,2);

\filldraw[black] (-1,0) circle (2pt) node at (-.8,-.25) {$p_0$};
\filldraw[blue] (.7,1.7) circle (2pt) node[black, anchor=west]{$p$};
\node at (-.75,.75){$\ell_1$};
\node at (1.4,1) {$\ell_2$};
\draw[red, thick] (0,2) --(3,-1);
\end{tikzpicture}
}
\hspace{5em}
\subfloat[$Z_{4,3}$]{
\begin{tikzpicture}
\node at (-.75,.75){$\ell_1$};
\node at (1.4,1) {$\ell_2$};
\draw[red, thick] (-2,-1) -- (1,2);
\filldraw[black] (-1,0) circle (2pt) node at (-.8,-.25){$p_0$};
\draw[red, thick] (0,2) --(3,-1);
\filldraw[blue] (2,0) circle (2pt) node[black, anchor=north]{$p$};

\end{tikzpicture}
}
\caption{Typical members of cycles on the extremal rays of $\Psef_4(X_{2,2})$. Fixed elements are black, varying elements in color.\label{figure1}}}

\end{figure}
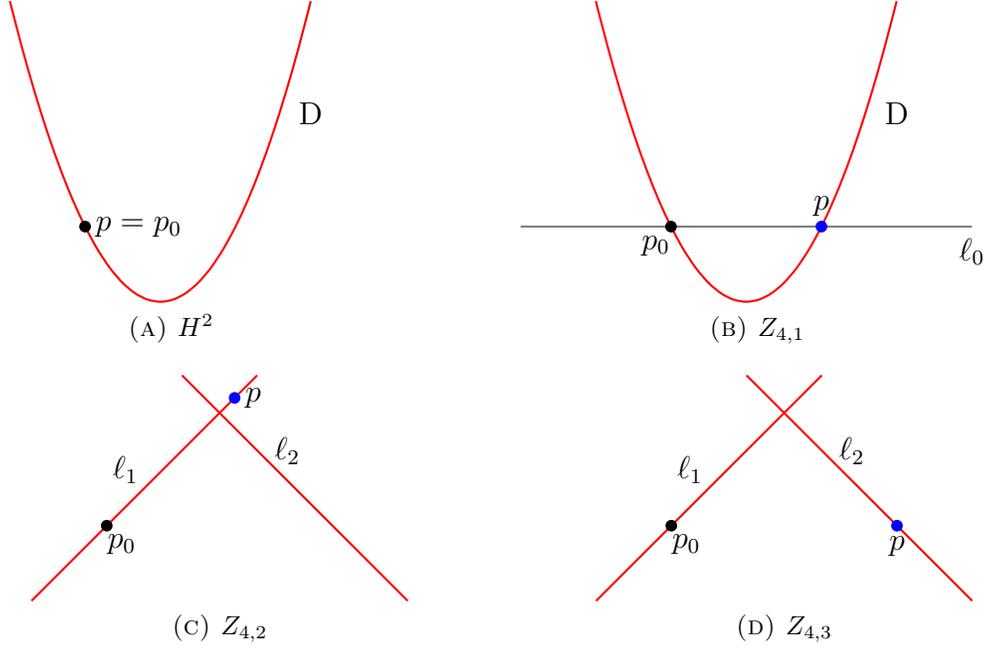

Let $Z_{3,1}$ and $Z_{3,2}$ be 3-cycles defined by 
\begin{align*}
Z_{3,1} &=\{(p, \ell_0+ \ell) \vert p\in \ell_0\}\\
Z_{3,2} &= \{(p,\ell_0+ \ell) \vert p \in \ell\}.
\end{align*}
\begin{theorem}[Corollary \ref{conicCorollary}]\label{conicCone}

Let $X_{2,2}$ be the universal conic in $\PP^2$.  

Then we have:
\begin{align*}
    \Psef_1(X_{2,2})&=\langle H^2\xi^3, [Z_{3,1}]\xi^2 \rangle\\
    \Psef_2(X_{2,2})&=\langle H^2\xi^2, [Z_{3,1}]\xi, [Z_{3,2}]\xi\rangle\\
    \Psef_3(X_{2,2})&= \langle H^2\xi,  [Z_{3,1}],[Z_{3,2}]\rangle\\
    \Psef_4(X_{2,2})&= \langle H^2, [Z_{4,1}], [Z_{4,2}],[Z_{4,3}]\rangle\\
    \Psef_5(X_{2,2})&=\langle \xi, H\rangle.
\end{align*}
Moreover, all the cycles appearing above span extremal rays. In particular, $\Psef_i(X_{2,2})$ is simplicial if and only if $i\neq 4$.
\end{theorem}
\begin{remark}
This article is inspired by the paper \cite{Ful11}, in which Fulger shows that the effective cones of cycles on a projective bundle $\PP E$ associated to a vector bundle $E$ on a curve $C$ are determined by the Harder-Narasimhan filtration of $E$. Unlike what we see in Theorem \ref{conicCone}, all of the cones $\Psef_i(\PP E)$ are simplicial since $N_i(\PP E)$ has rank at most 2 when $E$ is a vector bundle over a curve.
\end{remark}
In other cases, we determine some, but not all of the effective cones. Echoing \cite{CLO16}, the general tendency is for the complexity of $\Psef_i(X_H)$ with $H$ fixed to increase with $i$. Likewise, the complexity of $\Psef_i(X_H)$ with $i$ fixed tends to \emph{decrease} as $H$ is replaced with a more positive divisor. This latter tendency comes from the existence of a product map $X_{H_1}\times Y_{H_2}\ra X_{H_1+H_2}$ for any effective $H_1,H_2$ with $H_1, H_1+H_2$ basepoint free. This decreasing complexity as $H$ becomes more ample inspires the following.
\begin{conjecture}\label{stabilizationV2}
Suppose every cone $\Psef_i(X)$ is generated by a finite set of effective cycles. Then for any nonnegative integer $i$, and $H$ sufficiently ample(depending on $i$), we have the following
\begin{enumerate}
\item  $\Psef_i(X_H)$ is generated by a finite set of effective cycles. \item If $H'$ is a basepoint free cycle with $H'-H$ represented by an effective divisor $D'$, then the  map $\Psef_i(X_{H})\ra \Psef_i(X_{H'})$ induced by the map $X_H\ra X_{H'}$ sending $(p,D)$ to $(p,D+ D')$ is an isomorphism.
\end{enumerate}
\end{conjecture}
Theorem \ref{lowCodimension} immediately establishes Conjecture \ref{stabilizationV2} for $i<n$. We can also verify Conjecture \ref{stabilizationV2} in low degree on $\PP^2$.
\begin{theorem}\label{stableCases}
If $X= \PP^2$, then Conjecture \ref{stabilizationV2} holds for $i\leq 7$ and for $i=10$. 
\end{theorem}
Indeed, in the cases $i\leq 6$ and $d\geq 3$ or $i\in\{7,10\}$ and $d \geq 4$ we describe the cones $\Psef_i(X_{2,d})$ precisely in Theorem \ref{stableCasesPrecise}.

Outside of these cases, we cannot verify Conjecture \ref{stabilizationV2}, but Theorem \ref{628} gives upper and lower bounds on the cones  $\Psef_i(X_{2,d})$ for any $i$ with $d$ sufficiently large (depending on $i$).

This paper is organized as follows. 
In Section \ref{preliminaries} we set up notation and collect the properties of positivity and projective bundles we will need. 
In Section \ref{generalities} we prove general facts about the effective cones of cycles on universal hypersurfaces, including Theorem \ref{lowCodimension}. 
In Section \ref{projSpace} we restrict our attention to universal hypersurfaces in $\PP^n$ and prove Theorems \ref{conicCone} and \ref{stableCasesPrecise}, the latter of which immediately implies Theorem \ref{stableCases}. 
Section \ref{stableConeBounds} describes the best bounds we have on the cones $\Eff_i(X_{2,d})$ for $i\geq 8$, explicitly given as Theorem \ref{628}.
Finally, Section \ref{formulary} contains some calculations of the direct images of pseudoeffective cones under the multiplication map, still working in universal hypersurfaces on $\PP^n$.
\subsection*{Acknowledgements}
We would like to thank Izzet Coskun for many helpful conversations.
\section{Preliminaries}\label{preliminaries}
\subsection{Universal hypersurfaces}\label{universalHypersurfaces}
We work with a projective variety $X$ of dimension $n$ over an algebraically closed field of arbitrary characteristic.  We denote $\RR$-vector space of numerical $i$-cycles by $N_i(X)$, and the abstract dual of $N_i(X)$ by $N^i(X)$. 
Given an $i$-dimensional subvariety $Z\subseteq X$, we use $[Z]\in N_i(X)$ to denote the numerical cycle associated to $Z$.
If $X$ is smooth, Poincare duality gives a canonical isomorphism $N_i(X)\xrightarrow{\sim} N^{n-i}(X)$. This notation is essentially the same as in \cite{FL17}

Given a rank $r$ vector bundle $E$ on $X$, we define the projective bundle $\pi:\PP(E)\ra X$ as the space of rank 1 \emph{quotients} of $E$, following the convention of \cite{Laz04}. We then have the tautological sequence on $\PP(E)$
\[
0\ra S\ra \pi^*(E)\ra \OO_{\PP(E)}(1)\ra 0,
\]
and $\xi_E=c_1(\OO_{\PP(E)}(1))\in N^1(\PP(E))$ satisfies the projective bundle formula
\[
\xi^r-c_1(E)\xi^{r-1}+c_2(E)\xi^{r-2}-\cdots + (-1)^r c_r(E)=0.
\]
Let $H$ be an effective Cartier divisor class on $X$. Given a basepoint free linear series $V\subseteq H^0(X,H)$, we define define a vector bundle $E$ by the sequence
\[
0\ra \OO(-H)\ra V^*\otimes \OO\ra E\ra 0.
\]
Then $X_{V}=\PP(E)$ is the universal linear series associated to $V$. 
Throughout this article, we work in the special case $V=H^0(X,\OO(H))$, in which case we write $X_H=\PP(E)$ and say $X_H$ is the \emph{universal hypersurface} of class $H$.
We denote the class $\xi_E$ by $\xi$---or, where there might be confusion otherwise, $\xi_H$.
We also set $Y_H=\PP(H^0(X,H)^*)$, so $X_H$ admits a projection $\pi_2:X_H\ra Y_H$ in addition to the projection $\pi_1:X_H\ra X$. In a modular sense, $Y_H$ is the parameter space of hypersurfaces $D$ in $X$ of class $H$, and $X_H\subset X\times Y_H$ consists of pairs $(p,D)$ such that $p\in D$.
Let $\pi_1:X_H\ra X$ and $\pi_2:X_H\ra Y_H$ denote the projections.

If $\dim(Y_H)=r$, we have the projective bundle formula
\[
N^*(X_H)\cong \frac{\pi_1^*(N^*(X))[\xi]}{(\xi^{r}-H\xi^{r-1}+H^2\xi^{r-2}-\cdots +(-1)^r H^{r})}.
\]
In particular, every cocycle on $X_H$ is a linear combination of products $\xi^i \eta$, where $\eta$ is the pullback to $X_H$ of some cocycle on $X$.
Given a cocycle $\eta\in N^i(X)$, we will also use $\eta$ to denote the pullback $\pi_1^*(\eta)\in N^i(X_H)$.

 If $H_1$ and $H_2$ are effective Cartier divisor classes on $X$ such that $H_1$ and $H_1+H_2$ are basepoint free, there is a product map 
\[
\mu_{H_1,H_2}:X_{H_1}\times Y_{H_2}\ra X_{H_1+H_2}
\]
that sends $(p, D_1), D_2$ to $(p,D_1+D_2,)$.  Where it is unlikely to cause confusion, we use $\mu$ to denote any product map $\mu_{H_1,H_2}$.
\subsection{Covariant cycle notation and the product formula}
Given a numerical cycle $\eta\in N_i(X_H)$, we write its numerical class in a way that behaves well under pushforwards.  We write 
\[
\eta=(\eta_n,\ldots,\eta_0)=\sum_{0\leq j\leq n}\eta_j e_j,
\] where $\eta_j\in N_j(X)$ is the cycle corresponding to the map $ N^j(X)\ra \RR$  given by
\[
\alpha \mapsto\twopartdef{\pi_1^*(\alpha)\cdot \eta \cdot \xi^{i-j}}{i\geq j}{0}{i<j}. 
\]
So, for instance, we write $\xi=([X], H,0\ldots, 0)=[X]e_n+He_{n-1}$. We note that if $\eta=(\eta_n,\ldots, \eta_0)\in N_i(X_H)$ and $i\geq n+1$, then $\eta \xi\in N_{i-1}(X_H)$ also is given by $\eta\xi=(\eta_n,\ldots,\eta_0)$.

This notation is useful for writing down pushforwards of cycles under the product map $\mu:X_{H_1}\times Y_{H_2}\ra X_{H_1+H_2}$.
\begin{proposition}[The product formula]\label{productFormula}
Let $H_1$ and $H_2$ be effective Cartier divisor classes on $X$ such that $H_1$ and $H_1+H_2$ are basepoint free, so there exists a product map $\mu:X_{H_1}\times Y_{H_2}\ra X_{H_1+H_2}$.
 Let $\eta\in N_s(X_{H_1})$ be given by $\eta=(\eta_n,\ldots,\eta_0)$. 
 For any numerical $t$-cycle $\delta $ of degree $d$ on $Y_{H_2}$, we have that the $s+t$ cycle $\mu_*(\eta\times \delta)\in N_{s+t}(X_{H_1+H_2})$ satisfies the formula
\[
\mu_*(\eta \times \delta)=d\left( \binom{s+t-n}{t}\eta_n,\binom{s+t-n+1}{t}\eta_{n-1},\ldots, \binom{s+t}{t} \eta_0\right).
\]
\end{proposition}
\begin{proof}
This is an application of the push-pull formula. Set $H=H_1+H_2$, fix $\zeta\in N^i(X)$ with $i\leq s$ and let $\zeta'=\zeta\,\xi_H^{s+t-i}\in N^{s+t}(X_H)$. Then we have
\[
\mu_*(\eta\times \delta)\cdot \zeta'= \eta \times \delta \cdot \mu^*(\zeta').
\]
We have $\mu^*(\xi_H)= \xi_{H_1}+\xi_{H_2}$, where $\xi_{H_2}$ is the pullback of the hyperplane class on $Y_{H_2}$.
Likewise, we have  $\mu^*(\zeta)=\zeta$. Then we have
\[
\mu_*(\eta\times \delta)\cdot \zeta'= (\eta \times \delta)\cdot \zeta (\xi_{H_1}+\xi_{H_2})^{s+t-i}.
\]
The only term of $(\xi_{H_1}+\xi_{H_2})^{s+t-i}$ with nonzero pairing with $(\eta\times \delta)$ is the monomial $\binom{s+t-i}{t}\xi_{H_1}^{s-i}\xi_{H_2}^t$, and we have
\[
(\eta\times \delta)\cdot \binom{s+t-i}{t}\xi_{H_1}^{s-i}\xi_{H_2}^t=\binom{s+t-i}{t}d \eta_i\cdot \zeta. 
\]
\end{proof}

We can also pull back along morphisms $f:X_1\ra X_2$.  Given a basepoint free Cartier divisor class $H$ on $X_2$,  $f^*(H)$ is likewise basepoint free.  We have a pullback morphism $f^*:X_{2,H}\ra X_{1,f^*(H)}$ given by sending the class $\alpha \xi_H^t$  (with $\alpha\in N^s(X_2)$) to $f^*(\alpha) \xi_{f^*(H)}^t\in N^{s+t}(X_1)$.
\subsection{Positivity for cycles}\label{fl}
Let $X$ be a smooth projective variety of dimension $N$. The two most evident positive cones in the vector spaces $N_i(X)$and their duals $N^i(X)$ are:
\begin{itemize}
    \item The \emph{pseudoeffective cone} $\Psef_i(X)\subset N_i(X)\cong N^{n-i}(X)$, given as the closure of the cone spanned by the numerical classes of $i$-dimensional subvarieties of $X$.
    \item The \emph{nef cone} $\Nef^i(X)\subset N^i(X)\cong N_{n-i}(X)$, defined as the dual cone to $\Psef_i(X)$.
\end{itemize}
The nef cone and pseudoeffective cone of divisors on $X$ are well-understood and satisfy a number of useful properties. One such basic property is the following.
\begin{proposition}\label{nefDivisorProducts}
Suppose $D\in \Nef^1(X)$. 
If $\eta\in \Psef_i(X)$, then $D\cdot \eta\in \Psef_{i-1}(X)$, and if $\eta\in \Nef^i(X)$, then $D\cdot \eta\in \Nef^{i+1}(X)$.
\end{proposition}
\begin{proof}
Nefness is preserved by proper pullback \cite[Example 1.4.4]{Laz04}, so if $i:Z\ra X$ is a proper inclusion such that $i_*[Z]=\eta$, $i^*(D)$ is nef on $Z$,  hence pseudoeffective since nef divisors are pseuodeffective. So $D\cdot \eta\in \Psef_{i-1}(X)$ if $\eta$ is an integral effective $i$-divisor; since $\Psef_{i-1}(X)$ is closed and preserved by scaling by positive real numbers, we also have $D\cdot \eta\in \Psef_{i-1}(X)$ if we only assume $\eta\in \Psef_{i}(X)$. Likewise, if $\eta\in \Nef^i(X)$, then if $\beta \in \Psef_{i+1}(X)$, we have $D\cdot \beta\in \Psef_{i}(X)$ and $\beta \cdot (D\cdot \beta)>0$. So $\beta \cdot D\in \Nef^{i+1}(X)$.
\end{proof}
This property does not imply that the intersection of two nef cycles is nef.  For instance,  if $E$ is an elliptic curve with complex multiplication, then there are two nef two-cycles with negative intsersection on the product $E\times E\times E\times E$ \cite{DELV11}. To bound the effective cones of universal plane curves, we will use one of the refined notions of positivity given in \cite{FL17}.

We will also use the following basic result.
\begin{lemma}\label{finiteMaps}
If $f:Z'\ra Z$ a finite map of projective varieties, then $f_*\Psef_i(Z')=\Psef_i(Z)$.
\end{lemma}
\begin{proof}
Given any irreducible subvariety $W\subseteq Z$, $f^{-1}(W)\ra W$ is a finite map of some degree $d$, so $[W]=\frac{f_*([f^{-1}(W)])}{d}$.  
\end{proof}
In this article we use Lemma \ref{finiteMaps} precisely to make the identification
\[
\Psef_i(\mu(X_{H_1}\times Y_{H_2}))=\mu_*(\Psef_i(X_{H_1}\times Y_{H_2})).
\]

\begin{definition}\label{bpfdef}
A cycle $\eta\in N_i(X)$ is \emph{strongly basepoint free} if there are quasiprojective varieties $U,W$ and morphisms $p:U\ra W$ and $s:U\ra X$ such that $p$ is projective and surjective with fibers of dimension $i$, $s$ is flat, and $\eta$ is the class of $s_*([p^{-1}(w)])$, where $w$ is a general point of $W$. The \emph{basepoint free cone} $\BPF_i(X)\subset N_i(X)$ is the closure of the cone spanned by the classes of strongly basepoint free cycles.
\end{definition}
Fulger and Lehmann \cite{FL17} prove that $\BPF_i(X)$ is a subcone of $\Psef_i(X)\cap \Nef_i(X)$.  We will occasionally show a cycle is nef by showing it lies in $\BPF_i(X)$. More typically, we will mostly work with cycles that cannot quite be described as strongly basepoint free, but look basepoint free away from some controllable locus.
These too will satisfy some positivity properties.

\begin{proposition}\label{2.4}
Let $\eta\in N_i(X)$ be a numerical class, let $U$ be a variety, let $p:U\ra W$ be a proper morphism of relative dimension $i$ to a quasi-projective variety $W$ such that each component of $U$ surjects onto $W$, and let $s:U\ra X$ be a morphism. Let $B$ be a variety of codimension $c$ such that $s^{-1}(B)$ has codimension $c$ in $U$. Then $B\cdot \eta $ is effective.
\end{proposition}
\begin{proof}
Let $w\in W$ be a general point. The intersection $s^{-1}(B)\cap U_w$ has dimension $i-c$ if nonempty. Then $s_*(U_w)$ is an effective cycle of numerical class $\eta$ that intersects $B$ in dimension $i-c$, so $B\cdot \eta$ is effective by \cite[Proposition 7.1.a]{Ful84}.
\end{proof}

\section{Universal hypersurfaces in general}\label{generalities}
We continue to use the notation of Section \ref{preliminaries}. Let $X$ be a smooth projective variety of dimension $n$, let $H$ be a basepoint free divisor class on $X$, and set $\xi=\pi_2^*(\OO_{Y_H}(1))$.
Set $r=h^0(X,H)-1$, so $X_H$ has dimension $r+n-1$ and $Y_H$ has dimension $r$. 
\subsection{Positivity of cycles }
$X_H$ has several nef classes coming from its projections to $Y_H$ and $X$.
\begin{lemma}\label{easiestNefClasses}
If $\eta\in \Nef^i(X)$, then for all $j\geq 0$ the cycle $\eta \xi^j \in N^{i+j}(X_H) $ is nef.
\end{lemma}
\begin{proof}
Since $\eta\in N^i(X)$ is nef, its pullback to $X_H$ is also nef. Similarly, $\xi=c_1(\pi^*_2(\OO_{Y_H}(1)))$  is a nef divisor, so by Proposition \ref{nefDivisorProducts} $\eta\xi^j$ is nef as well.
\end{proof}
 When it applies, Theorem \ref{lowCodimension} asserts that if $i+j<n$, these are the \emph{only} nef cycles. We prove it now.
\begin{proof}[Proof of Theorem \ref{lowCodimension}]
Since $\xi$ is a nef divisor it follows from Lemma \ref{nefDivisorProducts} that 
\[
\Nef^i(X_H)\supseteq \sum_{0\leq j \leq i} \xi^{i-j}(\Nef^j(X)).
\]
 We prove the opposite inclusion dually by showing
\begin{equation*}
\Psef_i(X_H)\supseteq \sum_{0\leq j\leq i}  \Psef_j(X)e_j.
\end{equation*}
By assumption, $\Psef_j(X)$ is the closure of a cone spanned a collection of classes $\{\beta_\alpha\vert \alpha \in A\}$ such that each $\beta_\alpha$ is represented by some effective $j$-cycle $Z_\alpha$ supported on a divisor of class $H$. Fix a divisor $D_\alpha$ of class $H$ containing $Z_\alpha$, and let $Z'_\alpha\subset X_H$ consist of pairs $(p,D_\alpha)$ with $p\in Z_\alpha$. Then $Z'_\alpha$ is an effective cycle of class $\beta_\alpha e_j$ on $X_H$. So, since $\Psef_i(X_H)$ is itself closed, we have $\Psef_j(X)e_j\subseteq \Psef_i(X_H)$, and 
\[\Psef_i(X_H)\supseteq \sum_{0\leq j\leq i} \Psef_j(X)e_j.
\]
Dualizing this containment gives $\Nef^i(X_H)= \sum_{0\leq j \leq i} \xi^{i-j}(\Nef^j(X))$, as we were to show.
\end{proof}

\begin{remark}
 For Theorem \ref{lowCodimension} to apply, the hypotheses that $\Psef_i(X)$ be spanned by effective cycles supported on divisors of class $H$ is necessary.
For instance, if $X=\PP^1\times \PP^1$ and $F_1$, $F_2$ are the class of a fiber of the first and second projections respectively then $X_{2F_1}$ is a trivial $\PP^1$ bundle over $X$ and we have
\[
\Nef^1(X_H)=\langle \xi-F_1, F_1,F_2\rangle.
\]

\end{remark}
Beyond the cases handled by Theorem \ref{lowCodimension} or Lemma \ref{easiestNefClasses}, some numerical classes cannot be proven nef directly, but only positive away from some subcone of the cone of effective cycles. One construction proceeds as follows.
Let $\pi:\mathcal{Z}\ra B$ be a flat family of pairs $(Z_b,Z_b')\subset X\times X$ where $Z_b$ is a $m$-cycle contained in the $(m+1)$-cycle $Z_b'$.  Set
\[
W_b=\overline{\{(p,D)\in X_H\vert\dim(Z'_b \cap D)=m,\, Z_b\subset D,\, p\in Z_b'\setminus Z_b\}},
\]
and let $\eta=[W_b]$. Suppose $W_b$ has codimension $C$ in $X_H$, so $\eta \in N^C(X_H)$.
\begin{proposition}\label{sortaBPF}
Let $R\subset X_H$ be a $C$-cycle and suppose $\eta\cdot [R]<0$. Then, for some integer $c$ with $0\leq c<C$, there exists a cycle $R'\subset R$ of dimension $C-c$ such that for a general pair $(p,D)\in R'$, the locus of $b\in B$ such that $Z_b\subset D$ and $p\in Z_b'\setminus Z_b$ has codimension at most $C-c-1$.
\end{proposition}
\begin{proof}
In what follows we assume $B$ is irreducible because the general result follows from the irreducible case.
Suppose $R$ is an effective $C$-cycle with $\eta\cdot [R]<0$. Then $W_b\cap R$ must have positive dimension for any $b\in B$.  Let $P_b$ be the union of the positive dimensional components of $W_b\cap R$. Let $P\subset B\times R$ be the union of the $P_b$.  $P$ is locally closed of dimension at least $\dim B+1$. Pick one component $P'$ of $P$ and let $R'$ be the image of $P'$ in $R$. The fiber $P_r$ in $P'$ over a general point $r$ of $R'$ has dimension at least $\dim B+1-\dim R'$, so the preimage of $P_r$ in $B$ also has dimension at least $\dim B+1-\dim R'$. So, if $R'$ has codimension $c$ in $R$, so $R'$ has dimension $C-c$, and the codimension of the preimage of $P_r$ in $B$ is at most $\dim R'-1=C-c-1$.
\end{proof}
\subsection{The stable pseudoeffective cones}
If $H_2-H_1$ is effective, the product map $\mu:X_{H_1}\times Y_{H_2-H_1}\ra X_{H_2}$ induces pushforward isomorphisms $\mu_*:N_i(X_{H_1})\xrightarrow{\sim} N_i(X_{H_2})$ for $i\leq r-1$. This map preserves effectiveness for cycles, and leads to the following definitions:

For any $i$,  the \emph{stable space of numerical cycles} $N_i(X_{stab})$ is  defined by
\[
N_i(X_{stab})=\bigoplus_{0\leq j\leq \min(i,n)} N_j (X).
\]
For any $H$, there is a natural map $p_H:N_i(X_H)\ra N_i(X_{stab})$, coming from sending the cycle $\eta e_j$ to the cycle $\eta \in N_j(X)\subseteq N_i(X_{stab})$; in what follows, we will also use $\eta e_j\in N_i(X_{stab})$ to denote this cycle.  The map $p_H$ commute with the pushforward maps induced by $\mu$ and are isomorphisms if $i\leq r-1$. 

\begin{definition}
The \emph{stable pseudoeffective cone of $i$-cycles} $\Psef_i(X_{stab})\subset N_i(X_{stab})$ is the closure of the union of the cones $p_{H*}(\Psef_i(X_H))$ as $H$ ranges over all basepoint free classes on $X$.
\end{definition}
The following lemma determines many of the cycles in $\Psef_i(X_{stab}$.
\begin{lemma}\label{315}
Let $0\leq i<n$ and let $\beta \in \Psef_i(X)$ be effective. Then, for any $j\geq i$, there is an ample divisor $H$ such that $\beta e_i\in \Psef_j(X_H)$.
\end{lemma}
\begin{proof}
Given $\beta$, $j$, let $Z$ be a cycle on $X$ of class $\beta$, and let $H_1$ be a basepoint free class such that $Z$ is supported on some divisor $D_Z$ of class $H_1$.
Let $H_2$ be a basepoint free divisor class with $h^0(X,H_2)\geq j-i+1$. Let $Z_1\subset X_{H_1}$ consist of pairs $(p,D_Z)$ with $p\in Z$. Then $Z_1\in \Psef_i(X_{H_1})$ has class $\beta e_i$, and if $\Lambda\subset Y_{H_2}$ is any $(j-i)$-plane, the image of $Z_1\times \Lambda\subset X_{H_1}\times Y_{H_2}$ in $X_{H_1+H_2}$ has class a multiple of $e_i\beta$. 
\end{proof}
Lemma \ref{315} implies  a variant of Theorem \ref{lowCodimension} for stable cones.
\begin{corollary}
If $0\leq i< n$, then
\[
\Psef_i(X_{stab})=\sum_{0\leq j\leq i}\Psef_j(X)e_{j}.
\]
\end{corollary}
\begin{proof}
For any basepoint free $H$ we have
\[
\Nef^i(X_H)\supseteq \sum_{0\leq j\leq \min(i,n)} \xi^{i-j}\Nef^j(X),
\]
hence
\[
\Psef_i(X_H)\subseteq \sum_{0\leq j \leq i} \Psef_j(X) e_j.
\]
Since the above containment holds for \emph{any} $H$, we have
\[
\Psef_i(X_{stab})\subseteq \sum_{0\leq j \leq i} \Psef_j(X) e_j
\]

And by Lemma \ref{315}, if $\beta\in \Psef_j(X)$ is \emph{effective}, then for any $i\geq j$ $\beta e_j\in \Psef_i(X_H)$ for some sufficiently ample $H$, so $\Psef_i(X_{stab})$ contains $\beta e_j$ for all effective $\beta$. Since $\Psef_i(X_{stab})$ is closed, we have
\[
\Psef_i(X_{stab})\supseteq \sum_{0\leq j \leq i} \Psef_j(X) e_j.
\]
\end{proof}
\begin{corollary}\label{mostOfTheStableCone}
If $\Psef_i(X)$ is generated by finitely many effective cycles for every $i$, then for any $j\geq n-1$ and $H$ sufficiently ample (depending on $j$), $\Psef_j(X_H)$ contains $e_0 \Psef_0(X),e_1\Psef_1(X),\ldots,e_{n-1}\Psef_{n-1}(X)$. Moreover, $\langle e_0\Psef_0(X),\ldots,e_{n-1}\Psef_{n-1}(X)\rangle$ is exactly the subcone of $\Psef_j(X_H)$ of cycles that have 0 intersection with $H^n\xi^{j-n}$. 
\end{corollary}
\begin{proof}
By Lemma \ref{315}, for $H$ sufficiently ample, $\Psef_j(X_H)$ contains any cycle $\eta e_k$ for $\eta$ some extremal $k$-cycle on $X$. Since there are only finitely many extremal cycles total, if $H$ is sufficiently ample, $\Psef_j(X_H)$ contains all of them. Conversely, $N_j(X_H)$,  the cone $ \langle e_0\Psef_0(X),\ldots,e_{n-1}\Psef_{n-1}(X),e_n\Psef_n(X)\rangle$ is  the dual cone to \[
\langle \Nef^0(X)\xi^j, \Nef^1(X)\xi^{j-1},\ldots, \Nef^{n}(X)\xi^{j-n}\rangle,\] which is clearly a subcone of $\Nef^j(X_H)$. So the subcone of $\Psef_j(X_H)$ of cycles that have 0 intersection with $H^n\xi^{j-n}$ must be contained in $\langle e_0 \Psef_0(X),e_1\Psef_1(X),\ldots,e_{n-1}\Psef_{n-1}(X)\rangle$. 
\end{proof}
Based on Corollary \ref{mostOfTheStableCone}, we ask the following question.

\begin{question}
Suppose $\Psef_j(X)$ is generated by a finite collection of effective classes for each $j$, and fix $i$.  Is  $\Psef_i(X_{stab})$  polyhedral? Is there some $H$ such that  $\Psef_i(X_{stab})=\Psef_i(X_H)$?
\end{question}
By Corollary \ref{mostOfTheStableCone}, the most interesting part of the stable cone consist of cycles that have nonzero intersection with $H^n\xi^{j-n}$. 


\section{Universal hypersurfaces in projective space}\label{projSpace}
Let $X_{n,d}$ be the universal degree $d$ hypersurface on $\PP^n$, and let $H$ be the class of a hyperplane.  In this section, when we have a cycle
\[
\eta=(d_n [\PP^n], \ldots, d_0 H^n)= d_n[\PP^n]e_n+d_{n-1}He_{n-1}+\cdots+d_0 H^n e_0,
\]
we suppress the classes $H^i$ in notation:
\[
\eta=(d_n,\ldots,d_0)=d_ne_n+d_{n-1}e_{n-1}+\cdots+d_0e_0.
\]
We set $r=\dim(Y_{n,d})=\binom{n+d}{d}-1$, so $X_{n,d}$ has dimension $r+n-1$.
\begin{lemma}\label{projectiveSpaceBPFCycles}
If $i<n$, then the cycle $e_i\in N_{n-1}(X_{n,d})$ is  basepoint free, and hence nef.
\end{lemma}
\begin{proof}
Let $W$ be the parameter space of all reduced complete intersections of $n-i$ hypersurfaces of degree $d$ in $\PP^n$,  and let $U$ be the universal complete intersection over $W$, so the fiber $U_w$ of $U$ over a point $w\in W$ is the complete intersection in $\PP^n$ corresponding to $w$.  Let $U'$ consist of pairs $w,(p,D)\in W\times X_{n,d}$ satisfying $p\in U_w$ and $U_w\subset D$.  Then $U'$ admits a flat projection $s:U'\ra X_{n,d}$ and a projective map $p:U'\ra W$ through $U$.  Hence,  given $w\in W$, the numerical class of $s(p^{-1}(w))$ is strongly basepoint free.  $s(p^{-1}(w))$ consists of all pairs $(p,D)$ with $p\in U_w$ and $D$ in the $(n-i-1)$-dimensional linear series giving the complete intersection $U_w$, and is hence a $(n-1)$-cycle having intersection $\deg(U_w)=d^{n-i}$ with $H^i\xi^{n-i-1}$ and intersection $0$ with all other $H^j\xi^{n-j-1}$. Hence $[s(p^{-1}(w))]=d^{n-i}e_i$ and $e_i$ is basepoint free.
\end{proof}

\begin{proposition}\label{lowDim}
 Let $0\leq i<n$. Then $\Psef_i(X_{n,d})=\Nef_i(X_{n,d})=\langle e_0,\ldots, e_i\rangle$ and $\Nef^i(X_{n,d})=\Psef^i(X_{n,d})=\langle \xi^i,H\xi^{i-1}, \ldots, H^i\rangle$. 

\end{proposition}
\begin{proof}
By Lemma \ref{projectiveSpaceBPFCycles}, the cone $\langle e_0, \ldots,e_i\rangle$ is contained in  $\Eff_i(X_{n,d})\cap \Nef_i(X_{n,d})$.  Likewise, we have $\langle \xi^{i}, H\xi^{i-1},\ldots,H^i\rangle\subseteq \Nef^i(X_{n,d}) \cap \Eff^i(X_{n,d}))$ since its generators are products of the nef divisors $\xi,H$. Dualizing this second containment, we have
\[
\langle e_0, \ldots,e_i\rangle \supseteq \Eff_i(X_{n,d})+\Nef_i(X_{n,d}),
\]
hence the result.

\end{proof}
This proposition determines $\Psef_i(X_{n,1})$ for all $n$ and $i$. 
\begin{corollary}[Proposition \ref{hyperplaneCone}]\label{456}
If $0\leq i<n$, then
\[
\Nef^i(X_{n,1})=\Psef^i(X_{n,1})=\langle H^i, H^{i-1}\xi,\ldots,\xi^i\rangle
\]
and
\[
\Nef_i(X_{n,1})=\Psef_i(X_{n,1})=\langle e_0,\ldots,e_i\rangle
\]
\end{corollary}
We next study effective cycles on $X_{n,2}$. We need some positivity information, provided by the following two results.
\begin{lemma}\label{367}
If $d\geq 2$, then on $X_{n,d}$  the cycle $\eta=H^{n-1}\xi-H^n$ is strongly basepoint free.
\end{lemma}
\begin{proof}
Given a fixed point $p_0$ on a fixed line $\ell_0$, $\eta$ is represented by the cycle $Z_{\ell_0,p_0}$ defined as the closure of the set of all $(p,D)$ such that $D$ contains $p_0$ and $p \in \ell_0\setminus p$. Letting $B$ be the parameter space of pairs $(\ell_0,p_0)$, we get a corresponding universal cycle $\mathcal{Z}\subset X_{n,d}\times B$.  Any pair $(p,D)$ is contained in a $2n-2$ dimensional set of the $Z_{\ell_0,p_0}$,  so by the miracle flatness theorem \cite[Theorem 23.1]{Mat89} the map $\mathcal{Z}\ra X_{n,d}$ is flat,  and $\eta$ is strongly basepoint free.
\end{proof}
\begin{lemma}\label{443}
If $d\geq 2$, then the class $nH^{n-1}\xi^{n+1}-(n+1)[pt]\xi^{n}$ is nef on $X_{n,d}$.
\end{lemma}
\begin{proof}
Let $B$ be the parameter space of rational normal curves $C$ in $\PP^n$ with $n+1$ distinct marked points $p_1,\ldots,p_{n+1}$.  Let $\mathcal{Z}\subset X_{n,d}\times B$ consist of pairs of pairs $((p,D),(C,(p_1,\ldots,p_{n+1})))$ such that $D\cap C$ contains $p_1+\ldots+p_{n+1}+p$. 
Through any $n+1$ distinct points not all lying in a hyperplane of $\PP^n$ there is a family of rational normal curves of dimension exactly $n$.  As a result,  the fiber dimension of $\mathcal{Z}$ over any pair $(D,p)\in X_{n,d}$ is $n^2+n-1$  unless $D$ is supported on a hyperplane,  in which case the fiber is empty.  So,  given any closed subvariety $W\subset X_{n,d}$, if  $b\in B$ is chosen general,  then $Z_b\cap W$ either has the expected dimension or is empty---the latter if and only if $W$ consists entirely of pairs with $D$ supported on a hyperplane.  So $nH^{n-1}\xi^{n+1}-(n+1)[pt]\xi^{n}$  is nef.
\end{proof}
\begin{proposition}\label{slightlyHigherDim}
Assume $n\geq 2$.
\[\Psef_i(X_{n,2})=	\left\{
		\begin{array}{ll}
			\langle e_0,\ldots,e_{n-1}, e_{n}+e_{n-1}\rangle\hfill  \mbox{if }n \leq i \leq 2n-1\\
			\langle e_0,\ldots,e_{n-2}, ne_n+(n+1)e_{n-1}, e_{n-1}+e_{n-2}, e_n+(n+1)e_{n-1} \rangle  \mbox{if } i=2n 
		\end{array}
	\right.
	\]
\end{proposition}
\begin{proof}
If $n\leq i\leq 2n-1$, then by the product formula and Corollary \ref{456} we have
 $\mu_*\Psef_i(X_{n,1}\times Y_{n,1})\subset\Psef_i(X_{n,2})$ is given by 
\[
\mu_*\Psef_i(X_{n,1}\times Y_{n,1})=\langle e_0,\ldots,e_{n-1}, e_n+e_{n-1}\rangle,
\]
where the cycle $e_n+e_{n-1}$ comes from the pushforward of the cycle $\xi^{2n-1-i}\in N_i(X_{n,1})$. (See Proposition \ref{1products}).  So we easily have
\[
\Psef_i(X_{n,2})\supseteq \langle e_0,\ldots,e_{n-1}, e_n+e_{n-1}\rangle.
\]
Since the divisors $\xi$ and $H$ are nef,  as is the cycle $H^{n-1}\xi-H^n$ by Lemma \ref{367},  we have
\[
\Nef^i(X_{n,2})\supseteq \langle H^n\xi^{i-n},H^{n-2}\xi^{i-n+2}, H^{n-3}\xi^{i-n+3},\ldots,\xi^i, H^{n-1}\xi^{i-n+1}-H^n\xi^{i-n}\rangle \subset N^i(X_{n,2}).
\]
Dualizing this containment gives
\[
\Psef_i(X_{n,2})\subseteq\langle e_0,\ldots,e_{n-1}, e_n+e_{n-1}\rangle.
\]
If $i=2n$,  by Proposition \ref{lowDim} and the product formula we have
\begin{align*}
\mu_*\Psef_{2n}(X_{n,1}\times Y_{n,1})=\langle & e_n+(n+1)e_{n-1}, ne_n+(n+1)e_{n-1}, \\ &2e_{n-1}+(n+2)e_{n-2}, ne_{n-1}+(n+2)e_{n-2},\ldots,ne_1+2ne_0\rangle.
\end{align*}
$\Psef_{2n}(X_{n,2})$ also contains the cycles $e_i$ with $i\leq n-2$. Indeed, for each $i<n-1$, $e_i$ is the class in $N^{\binom{i+2}{2}+n-i-1}(X_{n,2})$ of $Z$, where $Z$ is given by fixing an $i$-plane $\Lambda$ and setting
\[
Z=\{(p,D)\vert \Lambda\subset D, p\in \Lambda\}.
\] 
Z has dimension $\binom{n+2}{2}+n-2- (\binom{i+2}{2}+n-i-1)$. Since the inequality
\[
\binom{n+2}{2}+n-2- (\binom{i+2}{2}+n-i-1) \geq 2n
\]
holds for $i\leq n-2$, we hence have that $e_i\in \Eff_{2n}(X_{n,2})$ for $i\leq n-2$, and is represented by $\xi^{\dim(Z)-2n}[Z]$.  . So we have the containment 
\[
\Psef_{2n}(X_{n,2})\supseteq \langle e_0,\ldots,e_{n-2}, ne_n+(n+1)e_{n-1}, e_{n-1}+e_{n-2}, e_n+(n+1)e_{n-1} \rangle. 
\]
For the other direction, we require some positive $(2n)$-cocycles on $X_H$. The complete intersection cycles $\xi^{2n},\ldots, H^n\xi^n$ are all nef.  Set $\eta=H^{n-2}\xi^3-H^{n-1}\xi^2+H^n\xi$.  $\eta$ represents pairs $(p,D)$ such that $D$ contains some fixed line $\ell$ and is marked elsewhere on some fixed 2-plane $\Lambda$ containing $\ell$.  A typical quadric contains a $2n-5$ dimensional space of lines, and the only quadrics containing a larger space of lines are those containing a hyperplane; those instead contain a $2n-4$ dimensional space of lines. By Proposition \ref{sortaBPF}, if $R$ is any irreducible dimension $n+1$ subvariety of $X_{n,2}$ not supported on $\mu(X_{n,1}\times Y_{n,1})$, then $R\cdot \eta\geq  0$.  In particular, 
\[
\Eff_{2n}(X)=\mu_*(\Eff_{2n}(X_{n,1}\times Y_{n,1}))+\{\beta \in \Eff_{2n}(X_{n,2})\vert \beta \cdot \eta\geq 0\}.
\]
Finally, by Lemma \ref{443} the cycle $\eta_2=nH^{n-1}\xi^{n+1}-(n+1)H^n\xi^n$ is nef.  The dual cone of the cone  $\Delta=\langle \eta\xi^{n-1}, \eta_2, \xi^{2n},\ldots,\xi^nH^n\rangle\subset N^{2n}(X_{n,2})$ is the cone
 \[
  \Delta^*= \langle ne_n+(n+1)e_{n-1}+e_{n-2}, e_{n-1}+e_{n-2}, e_{n-2},e_{n-3},\ldots,e_0\rangle.
 \]
We have $\Psef_{2n}(X_{n,2})\subseteq \Delta^*+\mu_*(\Eff_{2n}(X_{n,1}\times Y_{n,1}))$,  by the nefness of every cycle in $\Delta$ except $\eta\xi^{n-1}$. So we have
\[
\Psef_{2n}(X_{2,n})\subseteq \langle e_0,\ldots,e_{n-2}, ne_n+(n+1)e_{n-1}, e_{n-1}+e_{n-2}, e_n+(n+1)e_{n-1} \rangle,
\] 
giving the result.
\end{proof}
This result and Proposition \ref{lowDim} are sufficient to give every effective cycle  on $X_{2,2}$.
\begin{corollary}[Theorem \ref{conicCone}]\label{conicCorollary}
The nontrivial pseudoeffective cones of $d$-cycles on $X_{2,2}$ are spanned by effective cycles and given by the following.
\begin{align*}
\Psef_1(X_{2,2}) &= \langle (0,1,0), (0,0,1)\rangle\\
\Psef_2(X_{2,2}) &= \langle (1,1,0), (0,1,0), (0,0,1)\rangle\\
\Psef_3(X_{2,2}) &=\langle (1,1,0), (0,1,0), (0,0,1).\rangle\\
\Psef_4(X_{2,2})&= \langle (2,3,0), (1,3,0), (0,1,1), (0,0,1)\rangle\\
\Psef_5(X_{2,2})&=\langle(1,2,0), (0,1,2)\rangle.
\end{align*}
\end{corollary}
\begin{remark}
In the notation of Theorem \ref{conicCone}, we have $[Z_{4,1}]=(0,1,1)$, $[Z_{4,2}]=(1,3,0)$, $[Z_{4,3}]=(2,3,0)$, $[Z_{3,1}]=(0,1,0)$, and $[Z_{3,2}]=(1,1,0)$. So Corollary \ref{conicCorollary} and Theorem \ref{conicCone} are equivalent.
\end{remark}

\subsection{Universal plane curves}
We now restrict our attention to $X_{2,d}$, with the eventual goal of proving Theorem \ref{stableCases}.
That theorem is an immediate consequence of the following.
\begin{theorem}\label{stableCasesPrecise}
Suppose $d\geq 3$ and $2\leq i\leq 6$ or $i\in\{7,10\}$ and $d\geq 4$. 
Then we have
\[
\Psef_i(X_{2,d})=\langle (1,\delta(i),0), (0,1,0), (0,0,1)\rangle
\]
where $\delta(i)$ is given by the following table.
\begin{center}
\begin{tabular}{c|c|c|c|c|c|c|c}
     i&2&3&4&5&6&7&10  \\
     \hline
     $\delta(i)$&1&1&1.5&2&2&2.4 &3
\end{tabular}
\end{center}
\end{theorem}
But before proving this theorem, we need to produce a handful of positive cycles on $X_{2,d}$.

Let $1\leq e<d$, and let $M=\binom{e+2}{2}$. For $c\leq M$, let $\eta_{e,c}\in N^c(X_{2,d})$ be the numerical class $eH\xi^{c-1}-(c-1)H^2\xi^{c-2}$. This class is effective. If $(e,d,c)\neq (1,2,3)$,  given a irreducible degree $e$ curve $C_0$ and a degree $c-1$ divisor $D$ on $C_0$, this class is represented by the codimension $c$ set $Z_{C_0, D}$ of pairs $(C,p)$ where $C\cap C_0$ contains $D+p$.  And if $(e,d)=(1,2)$ $\eta_{1,3}$ instead is represented by the three-cycle of conics containing and marked on a fixed line,  and is hence still effective. We also have the following.

\begin{lemma}\label{almostNef}
If $c<M$, then $\eta_{e,c}$ is strongly basepoint free on $X_{2,d}$. If $c=M$ and $Z$ is any cycle of dimension at least $M$ not supported on $\mu(X_m\times Y_{n-m})$, then $\eta_{e,c}\cdot Z$ is a pseudoeffective class.
\end{lemma}
\begin{proof}

Suppose $c\leq M$ and $(m,n)\neq (1,2)$.
Let $U$ consist of pairs $(\Gamma, C_0)$ where $\Gamma$ is a collection of $c-1$ distinct points imposing independent conditions on degree $m$ curves, and $C_0$ is a smooth curve of degree $m$ containing $\Gamma$ such that the tangent line through any $q\in \Gamma$ does not contain any other point in $\Gamma$.
Then given any point $p\in C_0$, we have that the length $c$ scheme $\Gamma+p$ supported on $C_0$ imposes independent conditions on degree $m+1$ curves in $\PP^2$; the union of a general line through some $q\in \Gamma$ and a general degree $m$ curve through $\Gamma\setminus \{q\}$ will not contain $p$. Since $n\geq m+1$, we also have that the scheme $\Gamma+p$ imposes independent conditions on degree $n$ curves.

Let $W$ be the set of pairs of pairs $((\Gamma,C_0), (C,p))\subset U\times X_n$ such that the scheme-theoretic intersection $C\cap C_0$ contains $p+\Gamma$.
Since the length $c$ scheme $\Gamma+p$ imposes independent conditions on degree $m+1$ curves, $W$ is a projective bundle over the space of triples $(p, \Gamma, C_0)$. This space of triples is in turn a smooth family of projective curves over $U$, so the overall projection $p:W\ra U$ is a smooth projective morphism, and hence $W$ is smooth. 

Now we consider the projection $s:W\ra X_n$.
If $c<M$, then $s$ has fibers of constant dimension $M-1$, and is hence flat by \cite[Theorem 23.1]{Mat89} since both $W$ and $X_n$ are smooth.
So $\eta_{m,c}$ is strongly basepoint free.
If $c=M$, then the fiber dimension of $s$ jumps (by 1) over curves containing a smooth degree $m$ component.
Suppose $Z\subset X_n$ is an irreducible subvariety of $X_n$ not supported on the image of $X_m\times Y_{n-m}$.
If $s^{-1}(Z)$ is nonempty, then it is equidimensional of dimension $c+M-1$, since the dimension of $Z\cap \mu(X_m\times Y_{n-m})$ is at most $c-1$, and the fiber dimension jumps by 1 over that locus.
Then, since $s(p^{-1}(u))$ has numerical class $\eta_{c,M}$ for any $u\in U$, the cycle $B\cdot \eta_{c,M}$ is effective by Proposition \ref{2.4}. 

Finally, if $m=1$, $n=2$ and $c=3$, we have that $H\xi^2-2H^2\xi\in N_3(X_2)=(0,1,0)$ is represented by the three-cycle of singular conics containing and marked on a fixed line.  This cycle is strongly basepoint free on the hypersurface $X_{2,1}\times Y_{2,1}\subset X_{2,2}$, and hence intersets all cycles not supported on $X_{2,1}\times Y_{2,1}$ positively.
\end{proof}
\begin{remark}
It is plausible that $\eta_{c,M}$ retains some positivity for $c>M$.
In particular, by the same argument as the proof, for any $c\leq mn$, $\eta_{c,M}$ will be positive on cycles of dimension at least $c$ that are dimensionally transverse to $\mu(X_m\times Y_{n-m})$. 
\end{remark}

We now can prove Theorem \ref{stableCasesPrecise}, and hence Theorem \ref{stableCases}.
\begin{proof}[Proof of Theorem \ref{stableCasesPrecise}]
Under the hypotheses of the theorem, except in the case $i=10$, by Proposition \ref{2Products} the cone $\langle (1,\delta(i),0), (0,1,0), (0,0,1)\rangle \subset N_i(X_{n,d})$ is exactly the pushforward of $\Psef_i(X_{2,2}\times Y_{2,d-2})$ under the product map, so it suffices to verify that every effective cycle is contained in that cone.  The nefness of the cycles below follow from Lemmas \ref{easiestNefClasses},  \ref{443}, and \ref{almostNef}.
\begin{itemize}
    \item If $i=2$, the cycles $\xi^2$, $H\xi-H^2$, and $H^2$ are nef, so
    \[
    \Psef_2(X_{2,d})\subseteq \langle (1,1,0), (0,1,0),(0,0,1)\rangle.
    \]
    \item If $i=3$, the cycles $\xi^3,$ $H\xi^2-H^2\xi$, and $H^2\xi$ are nef, so
    \[
    \Psef_3(X_{2,d})\subseteq \langle(1,1,0),(0,1,0),(0,0,1)\rangle.
    \]
    \item If $i=4$, the cycles $\xi^4$, $H^2\xi^2$ are nef, while the cycle $H\xi^3-2H^2\xi^2$ is nef on cycles not supported on the image of $X_{2,1}\times Y_{2,d-1}$, so by Lemma \ref{1products} we have
    \[
    \Psef_4(X_{2,d})\subseteq \langle (2,3,0), (0,1,0), (0,0,1)\rangle.
    \]
    \item If $i=5$, the cycles $\xi^5$, $H^2\xi^3$, and $2H\xi^4-4H^2\xi^3$ are all nef, so
    \[
    \Psef_5(X_{2,d})\subseteq \langle (1,2,0), (0,1,0), (0,0,1)\rangle.
    \]
    \item If $i=6$, the cycles $\xi^6$, $H^2\xi^4$, and $2H\xi^5-4H^2\xi^4$ are all nef, so
    \[
        \Psef_6(X_{2,d})\subseteq \langle (1,2,0), (0,1,0), (0,0,1)\rangle.
        \]
    \item If $i=7$, the cycles $\xi^7$, $H^2\xi^5$ are nef, while $2H\xi^6-5H^2\xi^5$ is positive on cycles not supported on the image of $X_2\times Y_{n-2}$. so by Proposition \ref{2Products} we have
    \[
\Psef_7(X_{2,d})\subseteq \langle (5,12,0), (0,1,0), (0,0,1)\rangle.
    \]
\end{itemize}
In the case $i=10$, we instead need to push forward from $X_{2,3}\times Y_{2,d-3}$ to find the effective cycles.  By the $i=5$ case, we have $\Psef_5(X_{2,3})=\langle (1,2,0),(0,1,0),(0,1,0)\rangle$. Pushing forward to $\Psef_{10}(X_{2,d})$ using the product map, we have that $(0,1,0), (0,0,1)\in \Psef_{10}(X_{2,d})$.  Likewise, $(1,3,0)$ is the image of the fundamental cycle of $X_{2,3}$ under the product map.  So we have
\[
\Psef_{10}(X_{2,d})\supseteq \langle (1,3,0), (0,1,0), (0,0,1)\rangle.
\]
Conversely, by Lemma \ref{almostNef} the cycle $\eta_{3,10}=3H\xi^9-9H^2\xi^8$ is nef if $d\geq 4$, so the containment above is in fact an equality.
\end{proof}
Theorem \ref{stableCasesPrecise} describes every nontrivial cone of cycles on $X_{2,3}$ except $\Psef_i(X_{2,3})$ with $i\in \{7,8,9\}$. By Proposition \ref{lowDim}, the effective cone of divisors on $X_{2,3}$ is spanned by $H$ and $\xi$.  When $i\in\{7,8\}$, the pseudoeffective cone of cycles is not known precisely, but we have the following bounds.
\begin{proposition}
\begin{align*}
\langle (5,12,0),(1,6,0),(0,1,1),(0,0,1)\rangle \subseteq&\Psef_7(X_{2,3}) \subseteq \langle(5,12,0),(0,1,0),(0,0,1)\rangle\\
\langle (2,7,0),(5,14,0),(0,1,2),(0,0,1)\rangle \subseteq &\Psef_8(X_{2,3})\subseteq \langle(2,5,0),(0,1,0),(0,0,1)\rangle
\end{align*}
\begin{proof}
By Proposition \ref{1products},
\[
\mu_*\Psef_7(X_{2,1}\times Y_{2,2})=\langle (1,3,0),(5,12,0),(0,5,14)\rangle
\]
and by Proposition \ref{2Products},  
\[
\mu_*\Psef_7(X_{2,2}\times Y_{2,1})= \langle (1,3,0),(5,12,0),(0,5,14)\rangle.
\]
Since $\Psef_7(X_{2,2})$ also contains the effective cycles $H\xi^2-2H^2\xi= (1,1,0) $ and $H^2\xi = (0,0,1)$,  we have
\[
\langle (5,12,0),(1,6,0),(0,1,1),(0,0,1)\rangle \subseteq \Psef_7(X_{2,3})
\]
Likewise,  by Propositions \ref{1products} and \ref{2Products}, we have
\begin{align*}
\mu_*\Psef_8(X_{2,1}\times Y_{2,2})=\langle (2,7,0)\rangle
\mu_* \Psef_8(X_{2,2}\times Y_{2,1})= \langle (5,14,0)\rangle.
\end{align*}
Since $\Psef_8(X_{2,3})$ also contains the cycles $H\xi-H^2=(1,2,0)$ and $H^2=(0,0,1)$, we have 
\[
\langle (2,7,0),(5,14,0),(0,1,2),(0,0,1)\rangle \subseteq \Psef_8(X_{2,3}).
\]
The bound $\Psef_7(X_{2,3})\subset \langle (5,12,0),(0,1,0),(0,01)$ follows by the same argument as the $i=7$ case of Theorem \ref{stableCasesPrecise}. Finally, the bound $\Psef_8(X_{2,3})\subseteq\langle (2,5,0),(0,1,0),(0,0,1)\rangle$ follows from the nefness of $\xi^2,H^2$, and the nefness of $2H\xi-5H^2$ away from the image of $X_{2,2}\times Y_{2,1}$.
\end{proof}
\end{proposition}
\subsection{Stable pseudoeffective cones for universal plane curves}\label{stableConeBounds}
We now study $\Psef_i(X_{2,d})$ for $d$ large, and in particular the behavior of $\Psef_i(\PP^2_{stab})$. In two cases,  unambiguously extremal effective cycles are relatively easy to describe.
First, the cycle  $(0,0,1)$ is effective in $\Psef^c(X_{2,d})$ for all $c\geq 2$, since it is the class $H^2\xi^{c-2}$.
Second, the effectiveness of the class $(0,1,0)\in N^c(X_{2,d})$ is relatively well understood:
\begin{proposition} \label{010}
If $c\geq d+1$, the class $(0,1,0)\in N^c(X_{2,d})$ is represented by an effective cycle. Conversely, if $c \leq d$, then no multiple of $(0,1,0)$ is represented by an effective cycle.
\end{proposition}
\begin{proof}
If $c=d+1$, then $(0,1,0)\in N^c(X_{2,d})$ is represented by the subvariety $Z_{\ell_0}$ of all pairs $(p,D)$ containing and marked on some fixed line $\ell_0$.  For $c>d+1$, we have $(0,1,0)=[Z_{\ell_0}]\xi^{c-d-1}$.

Conversely,  suppose $Z\subset X_{2,d}$ is an irreducible variety with class $(0,k,0)\in N_i(X_{2,d})$.  Since $[Z]\xi^{i-2}H^2=0$,  there is some irreducible curve $C_0\subset \PP^2$ such that $p\in C_0$ for any $(p,D)\in Z$.  Likewise, since $[Z]\xi^i=0$, the image of $Z$ in $Y_{2,d}$ is a subvariety $B$ of dimension at most $i-1$. Since $\dim Z = i$ and $Z\subseteq \{(p,D)\vert p\in C_0, D\in B\}$, we must have
\[
Z= \{(p,D)\vert p \in C_0, D \in B\}.
\]
Hence, if $D\in B$, $D$ contains $C_0$.Containing $C_0$ imposes at least $d+1$ conditions on $B$, so $B$ has codimension at least $d+1$ in $Y_{2,d}$, and $Z$ has codimension at least $d+1$ in $X_{2,d}$.
\end{proof}
\begin{remark}
While $(0,1,0)\in N^{c}(X_{2,d})$ is not effective for $c\leq d+1$, except in cases covered above,  it is not clear whether it is a pseudoeffective class. We expect $(0,1,0)$ to be pseudoeffective only in the cases where Proposition \ref{010} gives its effectiveness.
\end{remark}
The above and Theorem \ref{stableCasesPrecise} hint toward the following refinement of Conjecture \ref{stabilizationV2}.
\begin{conjecture}\label{641}
For any $i$, there exists a rational number $\delta(i)$ and an integer $d(i)$ such that
\[
\Psef_i(\PP^2_{stab})=\Psef_i(X_{2,d(i)})=\langle (1,\delta(i),0),(0,1,0)(0,0,1)\rangle
\]
\end{conjecture}
Except in the cases covered by Theorem \ref{stableCasesPrecise},  Conjecture \ref{641} has not been verified. The best bounds we have are the following.
\begin{theorem}\label{628}
Let $i>1$ be an integer, and define $m$, $d_0$, $\delta_{\max}(i)$,  and $\delta_{\min}(i)$ as follows:
\begin{align*}
m&=\lfloor \frac{-3+\sqrt{1+8i}}{2}\rfloor\\
d_0& = m+\lceil \frac{-3+\sqrt{17+8m}}{2}\rceil\\
\delta_{\max}(i)&=\twopartdef{\min\left(\frac{2i-2}{m+3},m+1\right)}{i\geq 3}{1}{i=2}\\
\delta_{\min}(i)&=\max\left(\frac{m+3}{2}, \frac{i-1}{m+1}\right).
\end{align*}
For all $i\geq 2$ and $d\geq d_0$ we have
\begin{equation}\label{630}
\langle (1, \delta_{\max}(i),0),(0,1,0),(0,0,1)\rangle \subseteq \Psef_i(X_{2,d}),
\end{equation} 
and for $i\geq 8$ (and any $d$) we have
\begin{equation}\label{634}
\Psef_i(\PP^2_{stab}) \subseteq \langle (1,\delta_{\min}(i),0),(0,1,0),(0,0,1)\rangle.
\end{equation}

\end{theorem}
\begin{remark}
\begin{enumerate}
\item The containment (\ref{630}) is an equality for $2\leq i\leq 7$ and $i=10$ by Theorem \ref{stableCasesPrecise}. The containment (\ref{634}) fails for $i\in \{3,4,6,7\}$; indeed in those cases $\delta_{\min}(i)>\delta_{\max}(i)$.
\item $m$ is the greatest integer satisfying $\binom{m+2}{2}\leq i$, and $d_0$ is the least positive integer satisfying $m+2\leq \binom{d_0-m+2}{2}$.
Both inequalities are important in the proof. 
\end{enumerate}
\end{remark}
Before proving this proposition, we require one numerical lemma.
\begin{lemma}\label{650}
\begin{enumerate}
\item For $i\geq 2$, the functions $\frac{\delta_{\max}(i)}{i-1}$ and $\frac{\delta_{\min}(i)}{i-1}$ are non-increasing in $i$.
\item For $i\geq 8$, $\delta_{\max}(i)\geq \delta_{\min}(i)$.
\end{enumerate}
\end{lemma}
\begin{proof}
Since $m$ is the largest integer such that $m^2+3m\leq 2i-2$,  we have $(m+1)^2+3(m+1)> 2i-2$ as well.\\
\noindent (1): Both $\frac{\delta_{\max}(i)}{i-1}$ and $\frac{\delta_{\min}(i)}{i-1}$ are clearly non-increasing between $i$ and $i+1$ except possibly if $i=\binom{p+2}{2}-1$ and $i+1=\binom{p+2}{2}$ where $p$ is a positive integer; those are the places where $m$ increases. Those cases must be handled by a direct calculation. So set $i=\binom{p+2}{2}-1=\frac{p^2+3p}{2}$. We have
\[
\frac{\delta_{\max}(i)}{i-1}=\frac{p}{i-1}=\frac{2}{p+3},
\]
and
\[
\frac{\delta_{\max}(i+1)}{i}=\frac{2i-2}{(i-1)(p+4)}=\frac{2}{p+4}.
\]
Likewise,
\[
\frac{\delta_{\min}(i)}{i-1}=\frac{1}{p},
\]
and
\[
\frac{\delta_{\min}(i+1)}{i}=\frac{1}{p}.
\]

So both $\frac{\delta_{\max}(i)}{i-1}$ and $\frac{\delta_{\min}(i)}{i-1}$ are non-increasing between $i=\binom{p+2}{2}-1$ and $i=\binom{p+2}{2}$.  \\
\noindent (2): For any $m\geq 1$, we have $\frac{m+3}{2} \leq m+1$ and $\frac{i-1}{m+1} \leq \frac{2i-2}{m+3}$.  For $i\geq 10$, we have $m\geq 3$, implying
\[
\frac{2i-2}{m+3}\geq m \geq \frac{m+3}{2}.
\]
And for $i\geq 6$, we have $m\geq 2$, implying
\[
\frac{i-1}{m+1}\leq \frac{m+4}{2}\leq m+1.
\]
So for any $i\geq 10$ we have $\delta_{\max}(i)\geq \delta_{\min}(i)$. We have $\delta_{\min}(8)=\frac{5}{2}$, $\delta_{\max}(8)=\frac{14}{5}$, $\delta_{\min}(9)=\frac{8}{3}$, and $\delta_{\max}(9)=3$, so the lemma is true in those cases as well.
\end{proof}
\begin{proof}[Proof of Theorem \ref{628}]
We first show $(1,\delta_{\max}(i),0)\in \Psef(\PP^2_{stab})$.
Theorem \ref{stableCasesPrecise} establishes this containment for $i\leq 7 $, so we may assume $i\geq 8$. Set $m$ and $d_0$, as in the theorem statement, and let $d\geq d_0$ be an integer.  Set $M=\binom{m+2}{2}$.
 From the definition of $d_0$, we have  that 
\[
m+1 \leq \binom{d-m+2}{2}-1, 
\]
so since $i-M\leq m+1$, we have
\[
i-M\leq \binom{d-m+2}{2}-1, 
\]
Since $i\geq 3$, we have that $m\geq 1$ and $d\geq d_0\geq m+1$, and hence the spaces $X_{2,m}$ and $Y_{2,d-m-1}$ are defined. Fix a linear series $V\subset Y_{2,d-m}$  of dimension $i-M$, and let $Z$ be the image in $X_{2,d}$ of $X_{2,m}\times B$. Then, by the product formula, $Z$ is an effective cycle of numerical class
\[
[Z]=\left(\binom{i-2}{i-M}, m\binom{i-1}{i-M},0\right)
\]
This is a multiple of the cycle $(m+3,2i-2,0)$.

Likewise, given $D\in Y_{2,d-m-1}$, we have that the class of $Z'=\mu(X_{2,m+1}\times \{D\})$ in $\Psef_{\binom{m+3}{2}}(X_{2,d})$ is $(1,m+1,0)$,  so $[Z'] \xi^{\binom{m+3}{2}-i}$ is an effective cycle in $\Psef_i(X_{2,d})$ of class $(1,m+3,0)$.  Since both $(1,m+3,0)$ and $(m+3,2i-2,0)$ are in $\Psef_i(X_{2,d})$, we have $(1,\delta_{\max}(i),0)\in \Psef_i(X_{2,d})$.

We now prove that if $i\geq 8$, then (\ref{634}) holds, using induction on $i$.
First, note that if $i\geq \binom{d+2}{2}=\dim(X_{2,d})$, the containment holds trivially. 
In particular, if $d\geq m$, the containment holds. 
So in what follows we may assume $d\geq m$ and $i<\binom{d+2}{2}$.
Suppose (\ref{634}) has been established for all $\Psef_j(X_{2,d})$ with $8\leq j < i$.  We establish (\ref{634}) for $\Psef_i(X_{2,d})$.
Since $\xi^{i}$ and $H^2\xi^{i-2}$ are nef cycles on all $X_{2,d}$, this requires showing that both of the two cycles $2H\xi^ {i-1}-(m+3)H^2\xi^{i-2}$ and $(m+1) H\xi^{i-1}-(i-1)H^2\xi^{i-2}$ are nef. By Lemma \ref{almostNef}, the cycle $(m+1) H\xi^{i-1}-(i-1)H^2\xi^{i-2}$ is nef as long as
\[
i<\binom{m+3}{2},
\]
which holds by the observation that $m$ is the greatest integer such that $i\geq \binom{m+2}{2}$.  Likewise, the cycle $2H\xi^ {i-1}-(m+3)H^2\xi^{i-2}$ is a multiple of $mH\xi^{i-1}-(M-1)\xi^{i-2}$.  We have $i\geq M$, so this cycle is $\xi^{i-M}\eta_{m,M}$ in the notation of Lemma \ref{almostNef}.  We claim this cycle is nef if $i\geq 8$. By Lemma \ref{almostNef}, $\xi^{i-M}\eta_{m,M}$ is nef on effective cycles not supported on $\mu(X_{2,m}\times Y_{2,d-m})$, so we show $\xi^{i-M}\eta_{m,M}$ is nef on cycles supported on that locus as well.

If $i>M$ and $M\geq j\geq 8$, by the inductive hypothesis we have
\[
\Psef_j(X_{2,m})\subseteq \langle(1,\delta_{\min}(j),0), (0,1,0),(0,0,1)\rangle,
\]
so by the product formula, we have
\[
\mu_*(\Psef_j (X_{2,m})\times \Psef_{i-j}(X_{2,d-m}))\subseteq \langle (1,\frac{i-1}{j-1}\delta_{\min}(j),0),(0,1,0),(0,0,1)\rangle.
\]
And if $2\leq j\leq 7$, by Theorem \ref{stableCasesPrecise} we have
\[
\mu_*(\Psef_j (X_{2,m})\times \Psef_{i-j}(X_{2,d-m}))\subseteq \langle (1,\frac{i-1}{j-1}\delta_{\max}(j),0),(0,1,0),(0,0,1)\rangle.
\]
Combining and using Lemma \ref{650}, we hence have
\[
\mu_*(\Psef_i (X_{2,m}\times X_{2,d-m}))\subseteq \langle (1,\delta_{\min}(i),0),(0,1,0),(0,0,1)\rangle.
\]
 So $\xi^{i-M}\eta_{c,M}$ is nef on cycles supported on $\mu(X_{2,m}\times Y_{2,d-m})$, and therefore is nef in general.
   
 If $i=M$, we  note that $\Psef_M(X_{2,m})$ is spanned by the fundamental cycle, that is, 
\[
\Psef_M(X_{2,m})=\langle (1,m,0)\rangle.
\]
$\eta_{m,M}$ intersects this cycle non-negatively, and hence is nef on all cycles supported on $\mu(X_{2,m}\times Y_{2,d-m})$.
\end{proof}
\begin{corollary}
For any $i$, the stable cones $\Psef_i(\PP^2_{stab})$ satisfy the bounds
\[
\langle (1, \delta_{\max}(i),0),(0,1,0),(0,0,1)\rangle \subseteq \Psef_i(\PP^2_{stab})\subseteq   \langle (1,\delta_{\min}(i),0),(0,1,0),(0,0,1)\rangle.
\]
\end{corollary}
\subsection{Pushforwards of effective cones}\label{formulary}
In this section we write down formulas for the pushforward of some effective cones of cycles on the universal hypersurfaces $X_{n,d}$ under product maps $\mu$.
Every formula in this section is an application of the product formula, Proposition \ref{productFormula}.

We first handle products involving the universal hyperplane, $X_{n,1}$. By Proposition \ref{hyperplaneCone} we have
\begin{equation}\label{564}
\Psef_i(X_{n,1})=\twopartdef{\langle e_n+e_{n-1},e_{n-1}+e_{n-2},\ldots,e_{i+1-n}+e_{i-n}\rangle}{i\geq n}{\langle e_0,e_1,\ldots, e_i\rangle}{i\leq n-1}
\end{equation}
Applying the product formula to (\ref{564}) gives the following.
\begin{proposition}\label{1products}
Set $M = \dim(Y_{n,d}) = \binom{n+d}{d}-1$.  

If $0\leq i\leq n-1$, then
\[
\mu_*(\Eff_i(X_{n,1}\times Y_{n,d}))= \langle  e_0,\ldots,e_i\rangle.
\]
If $n\leq i\leq 2n-1$, then
\[
\mu_*(\Eff_i(X_{n,1}\times Y_{n,d}))=\langle e_0,\ldots,e_{n-1}, e_n+e_{n-1}\rangle.
\]
If $2n-1\leq i\leq M+n-1$, then
\[
\mu_*(\Eff_i(X_{n,1}\times Y_{n,d}))=\langle e_0,\ldots,e_{n-1} ne_n+(i-n+1)e_{n-1}\rangle.
\]
If $M+n \leq i\leq M+2n-1$, then
\begin{align*}
\mu_*(\Eff_i(X_{n,1}\times Y_{n,d})) =&\sum_{ i-M-n+1\leq j \leq n} \langle (i-j-M+1)e_j+(i-j+1)e_{j-1}\rangle \\
+&\sum_{ i-M-n+1\leq j \leq n} \langle n e_j +(i-j+1) e_{j-1}\rangle.
\end{align*}
\end{proposition}

The other cone of this sort we can determine is the pushforward of the effective cones of $X_{2,2}\times Y_{2,d}$.  
Applying the product formula to Corollary \ref{conicCorollary} gives the following.
\begin{proposition}\label{2Products}
Set $M=\dim(Y_{2,d})=\binom{d+2}{2}-1$. If $6\leq i\leq M+3$, then
\[
\mu_*(\Psef_i(X_{2,2}\times Y_{2,d}))=\langle (5,2(i-1),0), (0,1,0), (0,0,1)\rangle.
\]
Outside of this range, we have
\begin{align*}
\mu_*(\Psef_1(X_{2,2}\times Y_{2,d}))&=\langle(0,1,0), (0,0,1)\rangle\\
\mu_*(\Psef_2(X_{2,2}\times Y_{2,d}))&=\langle(1,1,0), (0,1,0), (0,0,1)\rangle\\
\mu_*(\Psef_3(X_{2,2}\times Y_{2,d}))&=\langle(1,1,0), (0,1,0), (0,0,1)\rangle\\
\mu_*(\Psef_4(X_{2,2}\times Y_{2,d}))&=\langle(2,3,0), (0,1,0), (0,0,1)\rangle\\
\mu_*(\Psef_5(X_{2,2}\times Y_{2,d}))&=\langle(1,2,0), (0,1,0), (0,0,1)\rangle\\
\mu_*(\Psef_{M+4}(X_{2,2}\times Y_{2,d}))&=\langle(5,2M+6,0), (1,M+3,0), (0,4, M+4), (0,0,1)\rangle\\
\mu_*(\Psef_{M+5}(X_{2,2}\times Y_{2,d})&=\langle(4,2M+8,0),  (5,2M+8,0),(0, 5, 2M+10)\rangle\\
\mu_*(\Psef_{M+6}(X_{2,2}\times Y_{2,d}))&=\langle(5,2M+10,0)\rangle.
\end{align*}
\end{proposition}
\bibliographystyle{plain}
\bibliography{citations}
\end{document}